\documentclass{amsart}

\newtheorem{theorem}{Theorem}[section]
\newtheorem{lemma}[theorem]{Lemma}

\theoremstyle{definition}

\newtheorem{example}[theorem]{Example}

\newtheorem{corollary}[theorem]{Corollary}
\theoremstyle{remark}
\newtheorem{remark}[theorem]{Remark}
\newtheorem{proposition}[theorem]{Proposition}

\numberwithin{equation}{section}

\usepackage{xcolor}
\usepackage{tikz}
\usepackage{pgfplots}
\pgfplotsset{compat=1.18}

\usepgfplotslibrary{colormaps}

\begin{document}

\title[Polyhedral norms and Hahn-Banach extension]{Polyhedral norms and smooth Hahn-Banach extension}

\author{Saikat Roy}

\address[Roy]{Department of Mathematics, School of Advanced Sciences, VIT-AP University, Beside AP Secretariat, Amaravati, 522241, Andhra Pradesh, India.}
\email{saikatroy.cu@gmail.com, saikat.roy@vitap.ac.in}

\thanks{The author thanks Professor Vladimir Kadets for a helpful discussion.}





\begin{abstract}
We find a necessary and sufficient condition for a smooth functional on a subspace to admit a norm-preserving smooth extension to the entire space in polyhedral norms. The characterization is geometric: such an extension exists if and only if the unique absolute norm-attaining point of the smooth functional is an extreme point of both the unit ball of the subspace and that of the ambient space. We show by example that such a result is not true in non-polyhedral norms, even under sufficiently strong hypothesis. Extremity of the norm preserving restrictions of extreme functionals are also discussed.
\end{abstract}

\maketitle

\section{Introduction}
\noindent Throughout the text, we work only with finite-dimensional real Banach space $X$ whose unit sphere is a polyhedron. We refer these spaces as polyhedral space. We denote the closed unit ball and unit sphere of $X$ by $B_X$ and $S_X$, respectively. The collection of smooth points lying in the unit sphere of $X$ is denoted by $sm(B_X)$, and the collection of all extreme points of $B_X$ is denoted by $ext(B_X)$.

\medskip

\subsection{Motivation}\hfill\\
\noindent Extension of linear functionals defined on a subspace to the whole space while retaining certain analytic properties has long been a central theme in functional analysis. The most fundamental result in this direction is the Hahn–Banach extension, which guarantees that every bounded linear functional defined on a subspace admits a norm-preserving extension to the entire space. Beyond norm preservation, several results address the extension of functionals with additional geometric features. A classical theorem due to Singer asserts that if $Y$ is a subspace of a normed linear space $Z$, then every extreme point of $B_{Y^*}$ admits an extension which is an extreme point of $B_{Z^*}$.

\begin{theorem}\label{Thm: Singer}\cite{S56, S65}
Let $Z$ be a (real or complex) normed linear space, $Y$ a linear subspace of $Z$ and $f$ an extreme point of the unit ball $B_{Y^*}$ of $Y$. Then there exists an extension of $f$ to an extreme point $\widetilde{f}$ of the unit ball $B_{Z^*}$ of $Z^*$.
\end{theorem}

\noindent The result follows from the Krein--Milman theorem and the fact that the
collection of all norm-preserving extensions of $f$ is a weak$^*$-compact
convex subset of $B_{Z^*}$, and hence possesses extreme points. This motivates the study of whether Singer’s theorem continues to hold when the notion of extremality is replaced by smoothness. More precisely, one may ask whether a smooth functional defined on a subspace admits a norm-preserving smooth extension to the whole space. The problem is finer than extremality, as Example~\ref{nonexample: nonpolyhedral} shows that such a lifting is not always possible in non-polyhedral norms, even under strong additional assumptions. Moreover, even within the class of polyhedral norms, not every smooth functional on a subspace admits a smooth norm-preserving extension, as demonstrated in
Example~\ref{nonexample: smooth}.

\medskip

\noindent The main objective of this short note is to investigate the existence of smooth norm-preserving extensions of smooth linear functionals in polyhedral spaces. We obtain necessary and sufficient conditions for the existence of smooth Hahn–Banach extensions in terms of extreme points, which is the main result of the article.
\begin{theorem}\label{smooth HB liftings} 
Let $X$ be a polyhedral space and $Y$ be a subspace of it. Then any smooth functional $f$ on $Y$ extends to a norm-preserving smooth functional $\widetilde{f}$ on $X$ if and only if $f$ attains its absolute norm at an extreme point of $B_Y$ which is also an extreme point of $B_X$.
\end{theorem}
\noindent The result characterizes the subspaces of a polyhedral space where every smooth linear functional admits smooth Hahn-Banach extension.
\begin{corollary}\label{first cor}
Let $X$ be a polyhedral space and $Y$ be a subspace of it. Then every smooth linear functional on $Y$ admits a smooth Hahn-Banach extension if and only if $ext(B_Y)\subset ext(B_X)$.
\end{corollary}
\noindent In another form, if $ext(B_Y)\subset ext(B_X)$ for a subspace $Y$ of $X$, then the norm preserving restrictions of smooth linear functionals are precisely the smooth linear functionals in $Y$.
\begin{corollary}\label{second cor}
Let $X$ be a polyhedral Banach space and let $Y$ be a subspace of $X$. Then 
\begin{equation}\label{Equation: smooth}
sm(Y^*)=\{f\in Y^*:~f=\widetilde{f}|_Y,~\widetilde{f}\in sm(X^*),~\|\widetilde{f}|_Y\|=1\}  
\end{equation}
if and only if $ext(B_Y)\subset ext(B_X)$.
\end{corollary}
\noindent We also discuss about the restriction of an extreme functional on $X$ to a subspace $Y$, which in general need not be an extreme point of $B_{Y^*}$, and it is quite expected, as the restriction may even fail to be norm-preserving. However, as shown in Example~\ref{norm pres. restr.}, even norm preserving restriction of an extreme linear functional may not remain extreme. We further obtain a necessary and sufficient condition under which the restriction of an extreme functional remains extreme. Although the characterization is fairly elementary, it provides important geometric insight. In the final section we show that our results facilitate the understanding of the smooth extension problem from a purely geometric consideration. 

.

\section{Preliminaries} 

\noindent A finite-dimensional real Banach space $X$ is said to be polyhedral if $B_X$ is a polyhedron; equivalently, $B_X$ has finitely many extreme points. If $X$ is polyhedral, then the exposed points of $B_X$ are precisely the extreme points of $B_X$. A subset $F$ of $B_X$ is called a face of $B_X$, if $\lambda x+(1-\lambda) y\in F$ implies that $x,y\in F$ for any $\lambda\in [0,1]$. The maximal faces of $B_X$ are called facets. The relative interior of a face $F$ is denoted by $relint~(F)$, which is defined by
\[
relint~(F)=\{x\in F:~\exists~ \varepsilon>0~\text{such that}~(x+\varepsilon B_X)\cap F\subset F\}
\]
We have a nice description of $relint~(F)$ for face of $B_X$ due to \cite[Theorem 6.9]{R70}.
\begin{proposition}\label{prop:smooth}
Let $X$ be a polyhedral space and $F$ be a face of $B_X$. Then,
\[
relint~(F)=\{\sum_{k=1}^n \lambda_k x_k:~~\sum_{k=1}^n \lambda_k=1, \lambda_k>0~\text{for all}~k\}.
\]
where $ext(F)=\{x_1,x_2,\dots,x_n\}$.
\end{proposition}

\noindent Our techniques are based on absolute norm attainment set and extreme point analysis. Given any member $f\in X^*$, the absolute norm attainment set of $f$ is denoted by $A_f$, and defined by
\[
A_f:=\{x\in S_X:~f(x)=\|f\|\}.
\]
Facets of $B_X$ can be characterized in terms of the absolute norm attainment sets of the extreme points of the dual unit ball.
\begin{proposition}
Let $X$ be a polyhedral space. Then, $F$ is a facet of $B_X$ if and only if $F=A_f$ for some unique $f\in ext(B_{X^*})$.
\end{proposition}

\noindent Evidently, $x\in sm(B_X)$ if and only if $x\in relint~(F)$ for a facet $F$ of $B_X$. Smoothness in the dual space can also be characterized in terms of the absolute norm attainment sets \cite{R24}.

\begin{lemma}\label{smooth}
Let $Z$ be a finite-dimensional Banach space. Then, $f\in sm(B_{Z^*})$ if and only if $A_f=\{z_0\}$ for some unit vector $z_0$. 
\end{lemma}

\medskip

\noindent If $X$ is a polyhedral space then so is $X^*$. Since $|ext(B_{X^*})|=k<\infty$, by the Krein-Milman Theorem 
\[
B_{X^*}=\{\sum_{j=1}^k \lambda_jf_j:~f_j\in ext(B_{X^*}),~\sum_{j=1}^k\lambda_j=1, ~\lambda_j\geq 0~\text{for all}~j\}.
\] 
If $x$ is a vector in $X$, then $\|x\|=1$ implies that there exists $f\in ext(B_{X^*})$ such that $f(x)=1$. Conversely, if $\max\{|f(x)|:~f\in ext(B_{X^*})\}=1$, then we have $\sum_{j=1}^k \lambda_jf_j(x)\leq 1$, for $f_j\in ext(B_{X^*})$ and $\lambda_j\geq 0$ for all $j$, proving $\|x\|=1$. Thus,  
\begin{equation}\label{unit sphere}
S_X = \{x\in X:~\max\{|f(x)|:~f\in ext(B_{X^*})\}=1\}.
\end{equation}
\noindent For detailed study of polyhedral spaces, the readers are referred to \cite{FHHSPZ01, GM72, SPBB19, R70}.


\section{Proofs of the Main Results and Further Discussion}

\noindent In this section we prove our main results and discuss few results on the restriction of extreme functionals.

\subsection{Proofs of the Main Results}

\begin{proof}[Proof of Theorem \ref{smooth HB liftings}]
First suppose that $f$ is a smooth functional on $Y$ that extends to a smooth functional  $\widetilde{f}$ on $X$ with the preservation of norm. Since $f$ is smooth, it attains its absolute norm at a unique point, i.e., $A_f=\{y_0\}$ for some unit vector $y_0$. Since  $\widetilde{f}$ is a Hahn-Banach extension of $f$, we have $A_f\subset A_{\widetilde{f}}$. At the same time $\widetilde{f}$ is a smooth linear functional, and thus it follows that $A_{\widetilde{f}}=\{y_0\}$. Therefore,  $\widetilde{f}$ is an exposing functional of $y_0$, and $y_0\in ext(B_X)$.

\medskip

\noindent Conversely, let $g$ be a smooth functional on $Y$ and $A_g=\{y_0\}$. Then $y_0\in ext(B_Y)$ and by the hypothesis $y_0\in ext(B_X)$. We now consider the (non-empty) collection
\[
S_{y_0}=\{\widetilde{f}\in ext(B_X):~\widetilde{f}(y_0)=1\}.
\]
Let $H$ be a subspace of $X$ defined as $H=\{x\in X:~\widetilde{f}(x)=0,~\widetilde{f}\in S_{y_0}\}$. Let $B_{y_0}=\{\widetilde{f}\in ext(B_{X^*}):~|\widetilde{f}(y_0)|<1\}$ and $C_{y_0}=\{\widetilde{f}\in ext(B_{X^*}):~|\widetilde{f}(y_0)|=1\}$. We observe that $C_{y_0}=S_{y_0}\cup (-S_{y_0})$ and $ext(B_{X^*})=B_{y_0}\cup C_{y_0}$. Also, the subspace $H$ is nothing but the intersection of kernels of all the linear functionals contained in $C_{y_0}$. If $H$ is a non-trivial subspace, we choose $u\in H$ and a real number $\lambda$ with sufficiently small absolute value such that $\max\{|\widetilde{f}(y_0\pm \lambda u)|:~\widetilde{f}\in B_{y_0}\}<1-\delta$ for some positive real number $\delta$. Also, it follows from the choice of $u$ that $|\widetilde{f}(y_0\pm \lambda u)|=1$ for each $\widetilde{f}\in C_{y_0}$. Thus,
\[
\max\{|\widetilde{f}(y_0\pm \lambda u)|:~\widetilde{f}\in ext(B_{X^*})\}=1.
\]
and we get $\|y_0\pm \lambda u\|=1$. However, then $y_0=\frac{1}{2}((y_0+\lambda u)+(y_0-\lambda u))$ is a convex combination of two unit vectors, a contradiction to the fact $y_0\in ext(B_X)$. This proves that $H$ must be trivial and shows $S_{y_0}$ contains a (algebraic) basis of $X^*$.

\medskip

\noindent Since $g\in sm(Y)$, $g\in relint~(W)$ for some facet $W$ of $B_{Y^*}$. Then by Proposition \ref{prop:smooth}
\[
g = \sum_{j=1}^p\lambda_jf_j,
\]
where 
\[
ext~(W)=\{f_1,f_2,\dots,f_p\}, \quad \sum_{j=1}^p\lambda_j=1, \quad \text{and}~ \lambda_j>0, ~\text{for all}~j.
\]
We now obtain a specific description of $W$. Let $\Pi:Y\to Y^{**}$ denote the canonical embedding. Then $\Pi(y_0)$ is an extreme point of $ext(B_{Y^{**}})$ and $\Pi(y_0)$ uniquely supports a facet of $B_{Y^*}$. Since $\Pi(y_0)(g)=1$, and $g\in relint~(W)$, we have $A_{\Pi(y_0)}=W$. Consequently, $f\in W$ if and only if $\Pi(y_0)(f)=f(y_0)=1$, which shows
\[
W=\{f\in B_{Y^*}:~f(y_0)=1\}.
\]
We now divide $S_{y_0}$ into two disjoint subsets
\begin{align*}
& S_{y_0}'=\{\widetilde{f}\in S_{y_0}:~\widetilde{f}|_Y=f,~f\in ext~(W)\},\\
& S_{y_0}''=\{\widetilde{f}\in S_{y_0}:~\widetilde{f}|_Y=f,~f\notin ext~(W)\}.
\end{align*}
Basically, $S_{y_0}'$ consists of members of $S_{X^*}$ that restricts to an extreme point of the facet $W$ of $B_{Y^*}$ and $S_{y_0}''$ represents the functionals on $X$ that restricts to an ordinary point (non-extreme) of the facet $W$.

\medskip

\noindent Let $I_j=\{\widetilde{f}\in S_{X^*}:~\widetilde{f}|_Y=f_j\}$ for each $j=1,2,\dots,p$. For convenience, assume
\[
I_j=\{\widetilde{f}_{1,j},\widetilde{f}_{2,j},\dots, \widetilde{f}_{n_j,j}\}, \qquad j=1,2,\dots,p,
\]
for natural numbers $n_1,n_2,\dots,n_p$. Let
\[
\mathcal{B}=(\bigcup_{j=1}^p I_j) \cup \mathcal{C},
\]
where $\mathcal{C}$ is a subset of $S_{y_0}''$ so that $\mathcal{B}$ contains a basis of $X^*$. Existence of such a set $\mathcal{B}$ is always guaranteed, since $S_{y_0}$ contains a basis of $X^*$. Anyway we may always choose $\mathcal{C} = \emptyset$, if the collection $(\bigcup_{j=1}^p I_j)$ already contains a basis of $X^*$. Let
\[
\mathcal{C}=\{\widetilde{g}_1,\widetilde{g}_2,\dots, \widetilde{g}_k\}~\qquad \text{and}  \qquad \widetilde{g}_i|_Y=g_i \quad \text{for each}~i=1,2,\dots,k.
\]
It now follows from the description of $W$ that $g_i\in W$ for each $i=1,2,\dots,k$.

\medskip

\noindent We now construct a Hahn-Banach extension $\widetilde{g}$ of $g$ such that $\widetilde{g}$ is a smooth linear functional on $X$. In particular, we shall show that the absolute norm attainment set of $g$ is a singleton set and use Lemma \ref{smooth} to conclude $\widetilde{g}$ is smooth. Let $\widetilde{g}:X\to \mathbb{R}$ be defined by
\[
\widetilde{g}(x)= \left[\sum_{j=1}^p\sum_{i=1}^{n_j}\alpha_j(\frac{1}{n_j}\widetilde{f}_{i,j})+\beta_{1}\widetilde{g}_{1}+\beta_{2}\widetilde{g}_{2}+\dots+\beta_{k}\widetilde{g}_k\right](x), \qquad x\in X,
\]

\noindent We now specify the choices for the scalars $\alpha_1, \alpha_2,\dots \alpha_p, \beta_{1}, \beta_{2}, \dots, \beta_k$:\\

\noindent Since $ext~(W)=\{f_1,f_2,\dots,f_p\}$, for each $r\in \{1,2,\dots, k\}$ there exist scalars $\mu_{1,r},\mu_{2,r}, \dots, \mu_{p,r}\in [0,1]$ with $\mu_{1,r}+\mu_{2,r}+ \dots +\mu_{p,r}=1$ such that
\[
g_{r}=\mu_{1,r}f_1+\mu_{2,r}f_2+ \dots +\mu_{p,r}f_p.
\]
For each $r$, choose $\beta_{r}>0$ such that
\[
\sum_{r=1}^{k}\beta_{r}\mu_{j,r}<\min\{\lambda_1,\lambda_2,\dots, \lambda_p\}, \quad \text{for each}~j\in \{1,2,\dots, p\}.
\]
Such a choice is always possible, since $\min\{\lambda_1,\lambda_2,\dots, \lambda_p\}>0$. Let
\[
\sum_{r=1}^{k}\beta_{r}\mu_{j,r}=t_j\lambda_j,\quad \text{for some}~t_j\in (0,1),~\text{for each}~j\in \{1,2,\dots, p\}.
\]
Also, let
\[
\alpha_j=(1-t_j)\lambda_j, \quad \text{for each}~j\in \{1,2,\dots, p\}.
\]
With such a specification, we claim that $\widetilde{g}$ is a smooth Hahn-Banach extension of $g$. Firstly, $\widetilde{g}$ is an extension of $g$. Observe that
\begin{align*}
\widetilde{g}\lvert_Y & = \left[\sum_{j=1}^p\sum_{i=1}^{n_j}\alpha_j(\frac{1}{n_j}\widetilde{f}_{i,j})+\beta_{1}\widetilde{g}_{1}+\beta_{2}\widetilde{g}_{2}+\dots+\beta_{k}\widetilde{g}_k\right]\Bigg\lvert_Y\\
& = \sum_{j=1}^p\alpha_jf_j+\beta_{1}g_{1}+\beta_{2}g_{2}+\dots+\beta_{k}g_k.
\end{align*}
It follows from the choice of the scalars $\alpha_1, \alpha_2,\dots \alpha_p, \beta_{1}, \beta_{2}, \dots, \beta_k$ that
\[
(\alpha_j+\sum_{r=1}^{k}\beta_{r}\mu_{j,r}) = (t_j\lambda_j+(1-t_j)\lambda_j) \qquad j\in \{1,2,\dots,p\}.
\]
Therefore, we get
\begin{align*}
\widetilde{g}\lvert_Y & = (\alpha_1+\sum_{r=1}^{k}\beta_{r}\mu_{1,r})f_1+(\alpha_2+\sum_{r=1}^{k}\beta_{r}\mu_{2,r})f_2+\dots+(\alpha_p+\sum_{r=1}^{k}\beta_{r}\mu_{p,r})f_p\\
& = (t_1\lambda_1+(1-t_1)\lambda_1)f_1+(t_2\lambda_2+(1-t_2)\lambda_2)f_2+\dots+(t_p\lambda_p+(1-t_p)\lambda_p)f_p\\
& = \lambda_1f_1+\lambda_2f_2+\dots+\lambda_pf_p\\
& = g.
\end{align*}

\medskip

\noindent Secondly, $\widetilde{g}$ is norm-preserving, since
\begin{align*}
1 =\|g\| & \leq \|\widehat{g}\|\\& \leq (\alpha_1+\sum_{r=1}^{k}\beta_{r}\mu_{1,r})+(\alpha_2+\sum_{r=1}^{k}\beta_{r}\mu_{2,r})+\dots+(\alpha_p+\sum_{r=1}^{k}\beta_{r}\mu_{p,r})\\
& =1.
\end{align*}
Thirdly, $\widetilde{g}$ attains its absolute norm only at $y_0$, since the collection $\mathcal{B}$ forms a basis of $X^*$, and
\[
A_{\widetilde{g}}=(\bigcap_{j=1}^{p}\bigcap_{i=1}^{n_j}A_{\widetilde{f}})\cap (\bigcap_{r=1}^{k}A_{\widetilde{g}_{r}}) = \bigcap_{l\in \mathcal{B}} (y_0+\ker~l)=\{y_0\}.
\]
Consequently, $\widetilde{g}$ is a smooth linear functional on $X$, and the proof is now complete.
\end{proof}

\noindent We now prove the remaining corollaries.

\begin{proof}[Proof of Corollary \ref{first cor}]
Suppose every smooth linear functional in $Y^*$ admits a smooth Hahn--Banach extension to the entire space $X$. Let $y_0\in ext(B_Y)$ and $f\in S_{Y^*}$ be the exposing functional of $y_0$. We know that $f$ admits a smooth Hahn--Banach extension $\widetilde{f}$ to $X$, and $A_f\subset A_{\widetilde{f}}$. However, since $\widetilde{f}$ is smooth, it follows from Lemma \ref{smooth} that $A_{\widetilde{f}}=\{y_0\}$. Therefore, $y_0\in ext(B_X)$, and we get $ext(B_Y)\subset ext(B_X)$. Next, assume that $ext(B_Y)\subset ext(B_X)$, and let $f\in sm(Y)$. Then $A_f\subset ext(B_X)$, and by Theorem \ref{smooth HB liftings}, $f$ extends to a smooth norm--preserving extension to $X$.
\end{proof}

\noindent Corollary \ref{second cor} is a direct consequence of Theorem \ref{smooth HB liftings} and Corollary \ref{first cor}.

\subsection{Restrictions of Extreme Functionals}

\noindent As noted in the introduction, a norm-preserving restriction of an extreme functional need not be extreme. We will shortly illustrate the situation in Example~\ref{norm pres. restr.} in the next section. Before that we record the following characterization for when the restriction of an extreme functional remains extreme.

\begin{theorem}\label{Restr. extr.}
Let $X$ be a polyhedral space and $Y$ be a subspace of it. Let $\widetilde{f}$ be an extreme point of $B_{X^*}$, then $\widetilde{f}$ restricts to an extreme point of $B_{Y^*}$ if and only if the absolute norm attainment set of $\widetilde{f}$ contains smooth points of $Y$.
\end{theorem}

\begin{proof}
Suppose $\widetilde{f}\in ext(B_{X^*})$ and $A_{\widetilde{f}}$ includes members of $sm(Y)$. Let $f$ be the restriction of $\widetilde{f}$ on $Y$. Then $A_f$ includes a smooth point $y_0$ contained in $relint~(F)$ of a facet $F$ of $B_Y$, and it follows that $F\subset A_f$. Since $A_f$ is a face of $B_Y$ and $F$ is a maximal face of $B_Y$, we get $A_f=F$. Consequently, $f\in ext(B_{Y^*})$. Conversely, if $\widetilde{f}\in ext(B_{X^*})$ restricts to a member $f$ of $ext(B_{Y^*})$, then $A_f$ supports a facet $M$ of $B_Y$. In particular, $relint~(M)\subset A_f \subset A_{\widetilde{f}}$, which shows that $A_{\widetilde{f}}\cap sm(Y)\neq \emptyset$, and the proof is now complete.
\end{proof}
\noindent The following corollary is an easy application of the above theorem and the fact that every extreme functional on a subspace admits an extreme extension to the entire space.
\begin{corollary}
Let $X$ be a polyhedral space and $Y$ be a subspace of it. Then\[ext(B_{Y^*})=\left\{f\in Y^*:~f=\widetilde{f}\lvert_{Y},~\widetilde{f}\in ext(B_{X^*}),~(A_{\widetilde{f}}\cap sm(Y))\neq \emptyset\right\}.\]
\end{corollary}

\section{Examples and Remarks}

\noindent We start with an example which shows that smooth functional on a subspace of a polyhedral space cannot be always extended to a smooth functional on the whole space with the preservation of norm.

\begin{example}\label{nonexample: smooth}
Let $X=(\mathbb{R}^3,\|\cdot\|_k)$, where 
\[
\|(x,y,z)\|_k=\max\{\frac{\sqrt{3}|y|+|x|+|\frac{|y|}{\sqrt{3}}-|x||}{2},|z|\}.
\]
Actually, $X=Z\oplus_\infty \mathbb{R}$, where $Z$ is a two-dimensional Banach space equipped with the (regular) hexagonal norm and the closed unit ball of $X$ is a hexagonal cylinder. Consider the subspace 
\[
Y=\{(x,y,z)\in X:~x=0\}.
\]
Consider the unit vector $(x_0,y_0,z_0)=(0,\frac{\sqrt{3}}{2},1)$ in $X$ which also lies in the unit sphere of $Y$. Let $f:Y\to \mathbb{R}$ be defined by
\[
f(x,y,z)= \frac{y}{\sqrt{3}}+\frac{z}{2}, \qquad (x,y,z)\in Y.
\]
It is not difficult to see that $\|f\|=1$ and $f$ attains its absolute norm only at $(x_0,y_0,z_0)$. Therefore, $f$ is a smooth point of $B_{Y^*}$. Any Hahn-Banach extension of $f$ is of the form 
\[
\widetilde{f}(x,y,z)=ax+\frac{y}{\sqrt{3}}+\frac{z}{2}, \qquad (x,y,z)\in X,
\]
for a real number $a$. If $a\neq 0$, then we choose $(\frac{sgn(a)}{2},\frac{\sqrt{3}}{2},1)\in S_X$ to conclude that 
\[
\|\widetilde{f}\|\geq |f(\frac{sgn(a)}{2},\frac{\sqrt{3}}{2},1)|>1.
\]
Therefore, $a=0$, since $\widetilde{f}$ is a Hahn-Banach extension of $f$. Thus, the unique Hahn-Banach extension of $f$ is given by 
\[
\widetilde{f}(x,y,z)=\frac{y}{\sqrt{3}}+\frac{z}{2}, \qquad (x,y,z)\in X.
\]
However, then
\[
\{(\alpha,\frac{\sqrt{3}}{2},1)\in S_X:~ \alpha\in [-1,1]\}\subset A_{\widetilde{f}},
\]
and by Lemma \ref{smooth}, $\widetilde{f}$ is not smooth.
\end{example}

\noindent However, the techniques developed in the article can be applied to detect the possibility of a smooth Hahn--Banach extension of smooth linear functional from a purely geometric perspective. We explain this fact pictorially in the following example.

\begin{example}
\noindent Case (A): Let $X=\ell_1^3$ and $Y=\{(x,y,z)\in X:~z=0\}$. Then $ext(B_Y)\subset ext(B_X)$, and by Corollary \ref{first cor} every smooth functional admits a smooth Hahn-Banach extension.

\medskip

\noindent Case (B): Let $X=\ell_\infty^3$ and $Y=\{(x,y,z)\in X:~z=0\}$. Then no extreme point of $B_Y$ is an extreme point of $B_X$. Therefore, absolute norm attainment set of any smooth functional on $Y$ does not include any extreme point of $B_X$. Therefore, by Theorem \ref{smooth HB liftings} no smooth linear functional on $Y$ admits a smooth norm--preserving extension to the entire space $X$.

\begin{figure}[h]

\begin{minipage}{0.52\textwidth}
\centering
\begin{tikzpicture}
\begin{axis}[
    view={20}{20},
    axis lines=none,
    axis equal image,
    xmin=-1.6, xmax=1.6,
    ymin=-1.6, ymax=1.6,
    zmin=-1.6, zmax=1.6,
    ticks=none,
    width=\textwidth,
]

\addplot3[dotted, thick] coordinates {(-1.6,0,0) (1.6,0,0)};
\node at (axis cs:1.75,0,0) {$x$};

\addplot3[dotted, thick] coordinates {(0,-1.6,0) (0,1.6,0)};
\node at (axis cs:0,1.75,0) {$y$};

\addplot3[dotted, thick] coordinates {(0,0,-1.6) (0,0,1.6)};
\node at (axis cs:0,0,1.75) {$z$};

\addplot3[thick] coordinates {
(1,0,0) (0,1,0) (-1,0,0) (0,-1,0) (1,0,0)
};

\addplot3[thick] coordinates {(0,0,1) (1,0,0)};
\addplot3[thick] coordinates {(0,0,1) (0,1,0)};
\addplot3[thick] coordinates {(0,0,1) (-1,0,0)};
\addplot3[thick] coordinates {(0,0,1) (0,-1,0)};

\addplot3[thick] coordinates {(0,0,-1) (1,0,0)};
\addplot3[thick] coordinates {(0,0,-1) (0,1,0)};
\addplot3[thick] coordinates {(0,0,-1) (-1,0,0)};
\addplot3[thick] coordinates {(0,0,-1) (0,-1,0)};

\addplot3[very thick] coordinates {
(1,0,0) (0,1,0) (-1,0,0) (0,-1,0) (1,0,0)
};

\end{axis}
\end{tikzpicture}

{\small Case (A)}
\end{minipage}

\begin{minipage}{0.42\textwidth}
\centering
\begin{tikzpicture}
\begin{axis}[
    view={120}{25},
    axis lines=none,
    axis equal image,
    xmin=-1.6, xmax=1.6,
    ymin=-1.6, ymax=1.6,
    zmin=-1.6, zmax=1.6,
    ticks=none,
    width=\textwidth,
]

\addplot3[dotted, thick] coordinates {(-1.6,0,0) (4,0,0)};
\node at (axis cs:3.75,0,0) {$x$};

\addplot3[dotted, thick] coordinates {(0,-1.6,0) (0,2,0)};
\node at (axis cs:0,2.2,0) {$y$};

\addplot3[dotted, thick] coordinates {(0,0,-1.6) (0,0,2)};
\node at (axis cs:0,0,2.2) {$z$};

\addplot3[thick] coordinates {
(-1,-1,-1) (1,-1,-1) (1,1,-1) (-1,1,-1) (-1,-1,-1)
};

\addplot3[thick] coordinates {
(-1,-1,1) (1,-1,1) (1,1,1) (-1,1,1) (-1,-1,1)
};

\addplot3[thick] coordinates {(-1,-1,-1) (-1,-1,1)};
\addplot3[thick] coordinates {(1,-1,-1) (1,-1,1)};
\addplot3[thick] coordinates {(1,1,-1) (1,1,1)};
\addplot3[thick] coordinates {(-1,1,-1) (-1,1,1)};

\addplot3[very thick] coordinates {
(-1,-1,0) (1,-1,0) (1,1,0) (-1,1,0) (-1,-1,0)
};

\end{axis}
\end{tikzpicture}

{\small Case (B)}
\end{minipage}

\end{figure}

\begin{center}
\begin{tikzpicture}
\begin{axis}[
    view={25}{20},
    axis lines=none,
    axis equal image,
    xmin=-2.8, xmax=2.8,
    ymin=-2.8, ymax=2.8,
    zmin=-2.8, zmax=2.8,
    ticks=none,
    width=11cm,
]

\addplot3[dotted, thick] coordinates {(-2.5,0,0) (2.5,0,0)};
\node at (axis cs:2.7,0,0) {$x$};

\addplot3[dotted, thick] coordinates {(0,-2.5,0) (0,2.5,0)};
\node at (axis cs:0,2.7,0) {$y$};

\addplot3[dotted, thick] coordinates {(0,0,-2.5) (0,0,2.5)};
\node at (axis cs:0,0,2.7) {$z$};

\coordinate (A) at (axis cs:1,0,0);
\coordinate (B) at (axis cs:-1,0,0);
\coordinate (C) at (axis cs:0,1,0);
\coordinate (D) at (axis cs:0,-1,0);
\coordinate (E) at (axis cs:0,0,1);
\coordinate (F) at (axis cs:0,0,-1);

\draw[thick] (A)--(C);
\draw[thick] (A)--(D);
\draw[thick] (A)--(E);
\draw[thick] (A)--(F);

\draw[thick] (B)--(C);
\draw[thick] (B)--(D);
\draw[thick] (B)--(E);
\draw[thick] (B)--(F);

\draw[thick] (C)--(E);
\draw[thick] (C)--(F);
\draw[thick] (D)--(E);
\draw[thick] (D)--(F);

\addplot3[
    mesh,
    domain=-2:2,
    y domain=-2:2,
    samples=2,
    samples y=2
]
({x},{y},{-y});

\addplot3[
    very thick,
    red
]
coordinates {
(1,0,0)
(0,0.5,-0.5)
(-1,0,0)
(0,-0.5,0.5)
(1,0,0)
};

\end{axis}

\end{tikzpicture}

\vspace{2mm}

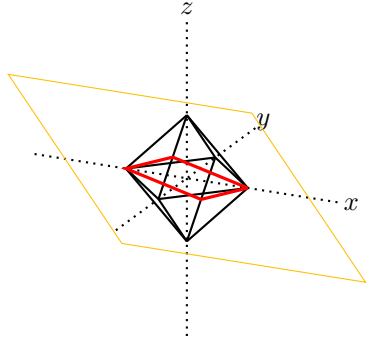
\small{Case (C): $B_{\ell_1^3}$ and its intersection with the plane $y+z=0$.}

\end{center}

\noindent Case (C): In this case $X=\ell_1^3$ and $Y=\{(x,y,z)\in X:~y+z=0\}$. Then $(1,0,0)$ and $(-1,0,0)$ are the only extreme points of $B_Y$ which are also extreme points of $B_X$. Therefore, a smooth linear functional $f$ on $Y$ admits a smooth Hahn-Banach extension if and only if $A_f=\{(1,0,0)\}$ or $A_f=\{(-1,0,0)\}$.
\end{example}

\noindent As noted earlier, restriction of an extreme functional may not remain extreme and we illustrate the situation in the following example.

\begin{example}\label{norm pres. restr.}
Let $X$ be the space as in Example \ref{nonexample: smooth}. Then $ext(B_X)$ is the collection $\{\pm(\frac{1}{2}, \frac{\sqrt{3}}{2},1), \pm(\frac{1}{2}, -\frac{\sqrt{3}}{2},1), \pm (1,0,1), \pm(\frac{1}{2}, \frac{\sqrt{3}}{2},-1), \pm(\frac{1}{2}, -\frac{\sqrt{3}}{2},-1), \pm (1,0,-1)\}$.  Let $\widetilde{f}:X\to \mathbb{R}$ be defined by
\[
\widetilde{f}(x,y,z)=x-\frac{1}{\sqrt{3}}y, \qquad (x,y,z)\in X.
\]
Since $\widetilde{f}$ attains its norm at some extreme points of $B_X$, and
\[
\max \{|\widetilde{f}(\mathbf{u})|:~\mathbf{u}\in ext(B_X)\}=1,
\]
we have $\|\widetilde{f}\|=1$. Moreover, $\widetilde{f}\in ext(B_{X^*})$, since $\widetilde{f}$ uniquely supports the facet $F=\text{convex hull}\{(\frac{1}{2}, -\frac{\sqrt{3}}{2},1), (\frac{1}{2}, -\frac{\sqrt{3}}{2},-1), (1,0,1),(1,0,-1)\}$ of $B_X$.

\medskip

\noindent We consider the subspace
\[
Y=\{(x,y,z)\in X:~z=2x\}.
\]
It is easy to see that $ext(B_Y)=\{\pm(\frac{1}{2}, \frac{\sqrt{3}}{2},1), \pm(\frac{1}{2}, -\frac{\sqrt{3}}{2},1)\}$.
Let $\widetilde{f}|_Y=f$ and we have
\[
f(x,y,z)= x-\frac{1}{\sqrt{3}}y, \qquad (x,y,z)\in Y.
\]
Let $g,h$ be linear functionals on $Y$ defined by
\[
g(x,y,z)=2x, \quad h(x,y,z)= -\frac{2}{\sqrt{3}}y, \qquad (x,y,z)\in Y.
\]
By checking the values of $g$ and $h$ on $ext(B_Y)$, we conclude that $\|g\|=\|h\|=1$. We observe that $f=\frac{1}{2}(g+h)$, which shows $f$ is not an extreme point of $B_{Y^*}$.
\end{example}

\noindent Finally, we provide a counter-example to Theorem \ref{smooth HB liftings} in a non-polyhedral norm, which proves that the problem of smooth Hahn--Banach extension is finer than the problem of extreme Hahn--Banach extension.

\begin{example}\label{nonexample: nonpolyhedral}
Let $X=(\mathbb{R}^3,\|\cdot\|_s)$, with the norm $\|\cdot\|_s$ is defined as \begin{align*}
\|(x,y,z)\|_s = \max & \{|x|,\frac{1}{2}(y^2+z^2)^\frac{1}{2}(2-\sigma)(1+\sigma)\\
& +  \frac{1}{2}\max\{|y|,|z|\}(2+\sigma)(1-\sigma)\}\\& \qquad \text{where}~\sigma=sgn(yz).
\end{align*}
\noindent Basically, $X=\mathbb{R}\oplus_\infty \ell_{2,\infty}$, where $\ell_{2,\infty}$ is a two-dimensional real Banach space equipped with the Day-James type norm \cite{NS06} $\|\cdot\|_{2,\infty}$, defined as
\[
\|(y,z)\|_{2,\infty}=\begin{cases}(y^2+z^2)^\frac{1}{2}, \quad \text{if}~yz\geq 0\\\max(|y|,|z|), \quad \text{if}~yz\leq 0\end{cases}
\]
Consider the $2$-dimensional space \[Y=\{(x,0,z):~x,z\in \mathbb{R}\}\]contained in $X$. The closed unit ball $B_Y$ of $Y$ is given by \[B_Y=\{(x,0,z):~\max\{|x|,|z|\}=1\}.\]

\noindent Notice that $X$ is not a polyhedral space. Consider any point of the form $(1,a,b)$ with $a^2+b^2=1$ and $ab\geq 0$. If 
\[
(1,a,b)=\frac{1}{2}[(1+\varepsilon_1, a+\varepsilon_2, b+\varepsilon_3)+(1-\varepsilon_1, a-\varepsilon_2, b-\varepsilon_3)], ~\text{for real numbers}~ \varepsilon_1, \varepsilon_2,\varepsilon_3,
\]
then in order to have
\[
\|(1+\varepsilon_1, a+\varepsilon_2, b+\varepsilon_3)\|=1=\|(1-\varepsilon_1, a-\varepsilon_2, b-\varepsilon_3)\|,
\]
we require that
\[
\max\{|1+\varepsilon_1|,[(a+\varepsilon_2)^2+(b+\varepsilon_3)^2]^\frac{1}{2}\}=1=\max\{|1-\varepsilon_1|,[(a-\varepsilon_2)^2+(b-\varepsilon_3)^2]^\frac{1}{2}\},
\]
which shows
\[
\varepsilon_1=0,~\text{and}~\varepsilon_2^2+\varepsilon_3^2 + 2 (a \varepsilon_2 + b \varepsilon_3)\leq 0,~\varepsilon_2^2+\varepsilon_3^2- 2 (a \varepsilon_2 + b \varepsilon_3)\leq 0~\text{since}~a^2+b^2=1,
\]
and thus,
\[
\varepsilon_1=\varepsilon_2=\varepsilon_3=0.\]This shows that\[\{(1,a,b):~a^2+b^2=1,~ab\geq 0\}\subset ext(B_X).
\]
Therefore, $ext(B_X)$ is not finite, and $X$ is not polyhedral.

\noindent Now, consider $(1,0,1)\in ext(B_X)$ and the linear functional $f:Y\to \mathbb{R}$, defined by 
\[
f(x,0,z)=\frac{1}{2}(x+z), \qquad (x,0,z)\in Y. 
\]
It is not difficult to see that $f$ attains its absolute norm only at $(1,0,1)$ and thus, $f\in sm(B_{Y^*})$. Now, we look for a smooth Hahn-Banach extension of $f$. Any Hahn-Banach extension $\widetilde{f}$ of $f$ is of the form
\[
\widetilde{f}(x,y,z)=\frac{1}{2}(x+z)+cy, \qquad (x,y,z)\in X,
\]
for some real number $c$. If $c > 0$, then for sufficiently small $\delta>0$, we choose
\[
(x_0,y_0,z_0)=(1,\frac{4c\delta}{(4c^2+1)}, (1-\frac{16c^2\delta^2}{(4c^2+1)^2})^\frac{1}{2})\in S_X.
\]
We claim that $\frac{z_0}{2}+cy_0>\frac{1}{2}$. Otherwise, $\frac{z_0}{2}+cy_0<\frac{1}{2}$, and since $y_0^2+z_0^2=1$, on simplifying, we have
\[
0 \leq y_0^2+4c^2y_0^2-4cy_0,
\]
which is equivalent to
\[
y_0\geq \frac{4c}{1+4c^2},
\]
which is impossible, since $\delta$ is sufficiently small. Therefore,  $\frac{z_0}{2}+cy_0>\frac{1}{2}$, and we get
\[
\widetilde{f}(x_0,y_0,z_0)>1,
\]
a contradiction that $\widetilde{f}$ is a Hahn-Banach extension of $f$. This shows that we must have $c\leq 0$. Again, if $c<0$, then
\[
\widetilde{f}(1,\alpha,1)>1, \quad \text{for}~\alpha\in (-1,0). 
\]
We observe that
\[
\|(1,\alpha,1)\|_s=1, \quad \text{for}~\alpha\in (-1,0),
\]
Consequently, we arrive to a contradiction that $\widetilde{f}$ is a Hahn-Banach extension of $f$. Therefore, $c=0$, and the unique Hahn-Banach extension of $f$ is given by
\[
\widetilde{f}(x,y,z)=\frac{1}{2}(x+z), \qquad (x,y,z)\in X.
\]
In that case we get\[\{(1,\alpha,1):~\alpha\in [-1,0]\}\subset A_{\widetilde{f}},\]and therefore, $\widetilde{f}$ fails to be smooth by Lemma \ref{smooth}.
\end{example}

\begin{remark}
Observe that the space $X$ considered in example \ref{nonexample: nonpolyhedral} is non-polyhedral, and the subspace $Y$ of $X$ is polyhedral with $ext(B_Y)=\{\pm(1,0,1), \pm(1,0,-1)\}$. The smooth functional $f$ considered on the subspace $Y$ attains it absolute norm at an extreme point $(1,0,1)\in ext(B_Y)\cap ext(B_X)$. However, $f$ does not admit a smooth Hahn-Banach extension on $X$, even though $Y$ is polyhedral. Thus, Theorem \ref{smooth HB liftings} is violated even for \emph{polyhedral subspaces} of a non-polyhedral space.
\end{remark}

\bibliographystyle{amsplain}

\end{document}